\documentclass[11pt,onepage]{amsart}

\usepackage{t1enc}
\usepackage[latin2]{inputenc}

\usepackage{amsmath}
\usepackage{amsthm}
\usepackage{amscd}
\usepackage{amssymb}
\usepackage{stmaryrd}
\usepackage[T1]{fontenc}

\usepackage{verbatim}
\usepackage{amscd}

\usepackage{url}
\usepackage{color}
\usepackage[small,nohug,
]{diagrams}
\diagramstyle[labelstyle=\scriptstyle]

\newarrow{Corresponds} <--->

\newarrow{Dashto} {}{dash}{}{dash}>

\usepackage[bitstream-charter]{mathdesign}
\usepackage{eucal}

\newcommand{\0}{\emptyset}
\newcommand{\om}{\omega}

\newcommand{\mc}{\mathcal}

\newcommand{\al}{\alpha}
\newcommand{\be}{\beta}
\newcommand{\ga}{\gamma}
\newcommand{\ka}{\kappa}
\newcommand{\lam}{\lambda}

\newcommand{\dom}{\mathrm{dom}}
\newcommand{\ran}{\mathrm{ran}}
\newcommand{\mf}{\mathfrak}
\newcommand{\mrm}{\mathrm}

\newcommand{\mbb}{\mathbb}
\newcommand{\acal}{\mathcal{A}}
\newcommand{\ical}{\mathcal{I}}
\newcommand{\eps}{\varepsilon}

\newcommand{\PP}{\mathbb{P}}
\newcommand{\fv}{\to}

\newcommand{\vd}{\Vdash}

\newtheorem*{claim}{Claim}

\newtheorem{thm}{Theorem}[section]

\newtheorem{prop}[thm]{Proposition}
\newtheorem{cor}[thm]{Corollary}
\newtheorem{fact}[thm]{Fact}

\theoremstyle{definition}
\newtheorem{que}[thm]{Question}
\newtheorem{df}[thm]{Definition}
\newtheorem{Def}[thm]{Definition}
\newtheorem{exa}[thm]{Example}

\newtheorem{rem}[thm]{Remark}

\title{Almost disjoint refinements and mixing reals}

\author{Barnab\'as Farkas}
\address{Kurt G\"odel Research Center for Mathematical Logic, Vienna, Austria}
\email{barnabasfarkas@gmail.com}

\author{Yurii Khomskii}
\address{University of Hamburg, Hamburg, Germany}
\email{yurii@deds.nl}

\author{Zolt\'an Vidny\'anszky}
\address{Alfr\'ed R\'enyi Institute of Mathematics, Budapest, Hungary}
\email{vidnyanszky.zoltan@renyi.mta.hu}

\thanks{The first author was supported by the Austrian
Science Fund (FWF) grant no. P25671, and the Hungarian National Foundation for Scientific Research grant nos. 83726 and 77476. The second author was supported by the Austrian
Science Fund (FWF) grant no. P25748. The third author is partially supported by the
Hungarian Scientific Foundation grants no.~104178 and no.~113047.}


\def\Ubf#1{{\baselineskip=0pt\vtop{\hbox{$#1$}\hbox{$\sim$}}}{}}

\begin{document}
\begin{abstract} We investigate families of subsets of $\om$ with almost disjoint refinements in the classical case as well as with respect to given ideals on $\om$. More precisely, we study the following topics and questions:

1) Examples of projective ideals.

2) We prove the following generalization of a result due to J. Brendle:
If $V\subseteq W$ are transitive models, $\om_1^W\subseteq V$, $\mc{P}(\om)\cap V\ne\mc{P}(\om)\cap W$, and $\mc{I}$ is an analytic or coanalytic ideal coded in $V$, then there is an $\mc{I}$-almost disjoint refinement ($\mc{I}$-ADR) of $\mc{I}^+\cap V$ in $W$, that is, a family $\{A_X:X\in\mc{I}^+\cap V\}\in W$ such that (i) $A_X\subseteq X$, $A_X\in \mc{I}^+$ for every $X$ and (ii) $A_X\cap A_Y\in\mc{I}$ for every distinct $X$ and $Y$.

3) The existence of perfect $\mc{I}$-almost disjoint ($\mc{I}$-AD) families; and the existence of a ``nice'' ideal $\mc{I}$ on $\om$ with the property: Every $\mc{I}$-AD family is countable but $\mc{I}$ is nowhere maximal.

4) The existence of $(\mc{I},\mrm{Fin})$-almost disjoint refinements of families of $\mc{I}$-positive sets in the case of everywhere meager (e.g. analytic or coanalytic) ideals. We show that under Martin's Axiom if $\mc{I}$ is an everywhere meager ideal and $\mc{H}\subseteq\mc{I}^+$ with $|\mc{H}|<\mf{c}$, then $\mc{H}$ has an $(\mc{I},\mrm{Fin})$-ADR, that is, a family $\{A_H:H\in \mc{H}\}$ such that (i) $A_H\subseteq H$, $A_H\in\mc{I}^+$ for every $H$ and (ii) $A_{H_0}\cap A_{H_1}$ is finite for every distinct $H_0,H_1\in\mc{H}$.

5) Connections between classical properties of forcing notions and adding mixing reals (and mixing injections), that is, a (one-to-one) function $f:\om\to\om$ such that $|f[X]\cap Y|=\om$ for every $X,Y\in [\om]^\om\cap V$. This property is relevant concerning almost disjoint refinements because it is very easy to find an almost disjoint refinement of $[\om]^\om\cap V$ in every extension $V\subseteq W$ containing a mixing injection over $V$.
\end{abstract}

\maketitle

\section{Introduction}
Let us begin with our motivations which led us to work on almost disjoint refinements and their generalizations. First of all, the following easy fact seems to be somewhat surprising (see also Proposition  \ref{refgen}):

\begin{fact}
If $\mc{H}\subseteq [\om]^\om(=\{X\subseteq\om:|X|=\om\})$ is of size $<\mf{c}$, then $\mc{H}$ has an \emph{almost-disjoint refinement} $\{A_H:H\in\mc{H}\}$, that is, (i) $A_H\in [H]^\om$ for every $H\in\mc{H}$ and (ii) $|A_H\cap A_K|<\om$ for every $H\ne K$ from $\mc{H}$.
\end{fact}

The following theorem due to B. Balcar and P. Vojt\'a\v{s} is probably the most well-know general result on the existence of almost-disjoint refinements.

\begin{thm} {\em (see \cite{BaVo})}
Every ultrafilter on $\om$ has an almost-disjoint refinement.
\end{thm}

B. Balcar and T. Paz\'ak, and independently J. Brendle proved the following theorem:
\begin{thm} (see \cite{BaPa}, \cite{Sou}) \label{togeneralize}
Assume that $V\subseteq W$ are transitive models and  $\mc{P}(\om)\cap V\ne\mc{P}(\om)\cap W$. Then $[\om]^\om\cap V$ has an almost-disjoint refinement in $W$ (where by {\em transitive model} we mean a transitive model of a ``large enough'' finite fragment of $\mathrm{ZFC}$).
\end{thm}

One of our main results is a generalization of this theorem in the context of ``nice'' ideals on $\om$, that is, we change the notion of {\em smallness} in the setting above by replacing {\em finite} with {\em element of an ideal $\mc{I}$}.

In order to formulate our generalization and to give a setting to our other related results, we have to introduce some notations and the appropriate versions of the classical notions.

Let $\mc{I}$ be an ideal on a countably infinite set $X$. We always assume that $[X]^{<\om}=\{Y\subseteq X:|Y|<\om\}\subseteq\mc{I}$ and $X\notin \mc{I}$. Let us denote by $\mc{I}^+=\mc{P}(X)\setminus \mc{I}$ the family of {\em $\mc{I}$-positive} sets, and by $\mc{I}^*=\{X\setminus A:A\in\mc{I}\}$ the {\em dual filter} of $\mc{I}$. If $Y\in\mc{I}^+$ then let $\mc{I}\upharpoonright Y=\{A\in\mc{I}:A\subseteq Y\}=\{B\cap Y:B\in\mc{I}\}$ be the {\em restriction} of $\mc{I}$ to $Y$ (an ideal on $Y$). If $X$ is clear from the contex, then the ideal of finite subsets of $X$ will be denoted by $\mathrm{Fin}$.

\begin{df}
 We say that a non-empty family $\acal\subseteq\ical^+$ is {\em $\ical$-almost-disjoint}
($\ical$-AD) if $A\cap B\in\ical$ for every distinct $A,B\in
\acal$. A family $\mc{A}\subseteq\mc{I}^+$ is $(\mc{I},\mrm{Fin})$-AD if $|A\cap B|<\om$ for every distinct $A,B\in\mc{A}$.
\end{df}
\begin{df}

Let $\mc{H}\subseteq\mc{I}^+$. We say that  a family $\mc{A}=\{A_H:H\in\mc{H}\}$ is an {\em $\mc{I}$-AD refinement} ($\mc{I}$-ADR) of $\mc{H}$  if (i) $A_H\subseteq H$, $A_H\in\mc{I}^+$ for every $H$, and (ii) $A_{H_0}\cap A_{H_1}\in\mc{I}$ for every distinct  $H_0,H_1 \in\mc{H}$ (in paticular, $\mc{A}$ is an $\mc{I}$-AD family). If $\mc{I}=\mathrm{Fin}$  we simply say AD-refinement (ADR).

We say that a family $\mc{A}=\{A_H:H\in\mc{H}\}$ is an {\em $(\mc{I},\mrm{Fin})$-AD refinement} ($(\mc{I},\mrm{Fin})$-ADR) of $\mc{H}$  if (i) holds and (ii)' $|A_{H_0}\cap A_{H_1}|<\om$ for every distinct $H_0,H_1\in\mc{H}$.
\end{df}

Notice that an ideal on a countably infinite $X$ can be regarded as a subset of the Polish space $2^X\simeq 2^\omega$ using a bijection between $X$ and $\omega$. Thus, it makes sense to talk about Borel, analytic, etc ideals and about certain descriptive properties of ideals, such as the Baire property or meagerness (it is easy to see that these properties do not depend on the choice of the bijection). In the past two decades the study of certain definable (e.g. Borel, analytic, coanalytic, etc.) ideals has become a central topic in set theory. It turned out that they play an important role in combinatorial set theory, and in the theory of cardinal invariants of the continuum as well as the theory of forcing (see e.g. \cite{Mazur}, \cite{So}, \cite{Farah}, \cite{Hrusak} and many other publications).

Now we can formulate our generalization of Theorem \ref{togeneralize}:
\begin{thm}\label{main}
Assume that $V\subseteq W$ are transitive models, $\om_1^W\subseteq V$, $\mc{P}(\om)\cap V\ne\mc{P}(\om)\cap W$, and $\mc{I}$ is an analytic or coanalytic ideal coded in $V$. Then there is an $\mc{I}$-ADR of $\mc{I}^+\cap V$ in $W$.
\end{thm}

We say that an ideal $\mc{I}$ on $X$ (where $|X|=\om$) is {\em everywhere meager} if $\mc{I}\upharpoonright Y$ is meager in $\mc{P}(Y)$ for every $Y\in\mc{I}^+$. In particular, analytic and coanalytic ideals are everywhere meager because their restrictions are also analytic and coanalytic, respectively, hence have the Baire property, and we can apply the following well-known  characterisation theorem (due to Sierpi\'{n}ski (1)$\leftrightarrow$(2), and Talagrand (2)$\leftrightarrow$(3), for the proofs see e.g.  \cite[Thm 4.1.1-2]{BaJu}).

\begin{thm}\label{t-char}
Let $\mc{I}$ be an ideal on $\om$. Then the following are equivalent: (1) $\mc{I}$ has the Baire property, (2) $\mc{I}$ is meager, and (3) there is a partition $\{P_n:n\in\om\}$ of $\om$ into finite sets such that $\{n\in\om:P_n\subseteq A\}$ is finite for each $A\in\mc{I}$.
\end{thm}

From now on, when working with partitions of a set, we always assume that every element of the partition is nonempty. From this theorem we can also deduce the following important corollary:

\begin{cor}\label{c-sAD}
If $\mc{I}$ is a meager ideal, then there is a perfect $(\mc{I},\mrm{Fin})$-AD family. In particular, if $\mc{I}$ is everywhere meager, then  there are perfect $(\mc{I},\mrm{Fin})$-AD families on every $X\in \mc{I}^+$.
\end{cor}
\begin{proof}
It is easy to define a perfect AD family $\mc{A}$ on $\om$ (e.g. consider the branches of $2^{<\om}$ in $\mc{P}(2^{<\om})$). Fix a partition $(P_n)_{n\in\om}$ of $\om$ into finite sets such that $\{n\in\om:P_n\subseteq A\}$ is finite for every $A\in\mc{I}$. For each $A\in \mc{A}$ let $A'=\bigcup\{P_n:n\in A\}\in \mc{I}^+$, and let $\mc{A}'=\{A':A\in \mc{A}\}$. Then $|A'\cap B'|<\om$ for every distinct $A,B\in\mc{A}$ hence $\mc{A}'$ is an $(\mc{I},\mrm{Fin})$-AD family. The function $\mc{P}(\om)\to\mc{P}(\om)$, $A\mapsto A'$ is injective and continuous hence $\mc{A}'$ is perfect.
\end{proof}

Concerning the reverse implications in Corollary \ref{c-sAD}, we prove the following.
\begin{thm}\label{t:reverse}$ $
\begin{itemize}
\item[(a)] The existence of a perfect $(\mc{I},\mrm{Fin})$-AD family does not imply that $\mc{I}$ is meager.
\item[(b)] If $\mf{b}=\mf{c}$ then there is an non-meager ideal $\mc{I}$ such that there are perfect $(\mc{I},\mrm{Fin})$-AD families on every $X\in\mc{I}^+$. Here $\mf{c}$ stands for the continuum and $\mf{b}$ for the {\em bounding number}, that is, $\mf{b}=\min\{|F|:F\subseteq\om^\om$ is $\leq^*$-unbounded$\}$ where $f\leq^* g$ iff the set $\{n\in\om:f(n)>g(n)\}$ is finite.
\item[(c)] There is an ideal $\mc{I}$ such that every $\mc{I}$-AD family is countable but $\mc{I}$ is nowhere maximal, that is, $\mc{I}\upharpoonright X$ is not a prime ideal for any $X\in\mc{I}^+$ (in particular, there are infinite $\mc{I}$-AD families).
\item[(d)] It is independent from $\mrm{ZFC}$ whether the example in (c) can be chosen as $\Ubf{\Sigma}^1_2$.
\end{itemize}
\end{thm}

Corollary \ref{c-sAD} has an easy but important application. Clearly, if $\mc{I}$ is an ideal on $\om$ then there is a family (e.g. $\mc{I}^+$) of size $\mf{c}$ which does not have any $\mc{I}$-ADR's. Conversely, we have the following very special case of results from  \cite{BHM} and \cite{BaSiVo}:
\begin{prop}\label{refgen}
If $\mc{I}$ is an everywhere meager ideal and $\mc{H}\in [\mc{I}^+]^{<\mf{c}}$, then $\mc{H}$ has an $\mc{I}$-ADR.
\end{prop}
\begin{proof}
Let $\mc{H}=\{H_\al:\al<\ka\}$. Applying Corollary \ref{c-sAD}, we can fix an $\mc{I}$-AD family $\mc{A}=\{A_\xi:\xi<\ka^+\}$ on $H_0$ and  for every $\be<\ka$ let $T_\be=\{\xi<\ka^+:H_\be\cap A_\xi\in\mc{I}^+\}$, furthermore let $R=\{\be<\ka:|T_\be|\leq \ka\}$ (we know that $0\notin R$). By induction on $\al\in\ka\setminus R$ we can pick a
\[ \xi_\al\in T_\al\setminus\Big(\bigcup_{\be\in R}T_\be\cup\big\{\xi_{\al'}:\al'\in\al\setminus R\big\}\Big)\] because $|T_\al|=\ka^+$ and $|\bigcup\{T_\be:\be\in R\}|\leq\ka$, and let $E_\al=H_\al\cap A_{\xi_\al}\in\mc{I}^+$. Then the family $\{E_\al:\al\in\ka\setminus R\}$ is an $\mc{I}$-ADR of $\{H_\al:\al\in\ka\setminus R\}$. We can continue the procedure on $\{H_\be:\be\in R\}$ because $E_\al\cap H_\be\in\mc{I}$ for every $\al\in\ka\setminus R$ and $\be\in R$.
\end{proof}

This proposition motivates the following:

\begin{que}\label{str}
Let $\mc{I}$ be an everywhere meager ideal and $\mc{H}\in [\mc{I}^+]^{<\mf{c}}$. Does $\mc{H}$ have an $(\mc{I},\mrm{Fin})$-ADR?
\end{que}

We answer this question, at least consistently:

 \begin{thm} \label{t:martin}
Assume $\mrm{MA}_\ka$ and let $\mc{I}$ be an everywhere meager ideal, then every $\mc{H}\in [\mc{I}^+]^{\leq\ka}$ has an $(\mc{I},\mrm{Fin})$-ADR.
\end{thm}

We also define new notions of mixing and injective mixing reals, and investigate connections between adding (injective) mixing reals  and classical properties of forcing notions (such as adding Cohen/random/splitting/dominating reals and the Laver/Sacks-properties).

\begin{Def}
Let $\PP$ be a forcing notion. We say that an $f\in \om^\om\cap V^\PP$ is a {\em mixing real} over $V$ if  $|f[X]\cap Y|=\om$ for every $X,Y\in [\om]^\om\cap V$.
If $f$ is one-to-one, then we call it an {\em injective mixing real} or  {\em mixing injection}.
\end{Def}

Our results are summarized in the following proposition.

\begin{prop} Let $\PP$ be a forcing notion.
\begin{itemize}
\item[(i)] If $\PP$ adds random  reals, then it adds mixing reals.
\item[(ii)] If $\PP$ adds dominating reals, then it adds mixing reals.
\item[(iii)] If $\PP$ adds Cohen reals, then it adds mixing injections.
\item[(iv)] If $\PP$ adds mixing injections, then it adds unbounded reals.
\item[(v)] If $\PP$ has the Laver-property, then it does not add injective mixing reals.
\end{itemize}
\end{prop}

Our paper is organized as follows. In Section \ref{prelim} we recall some notations and classical results of descriptive set theory we will need later.

The next two sections are focused on descriptive aspects of nice ideals and almost disjoint refinements. In Section \ref{projid} we present a plethora of examples of Borel and projective ideals on $\om$. In Section \ref{mainsec} we prove Theorem \ref{main} by modifying Brendle's proof of Theorem \ref{togeneralize}.

The next two sections contain rather combinatorial results. In Section \ref{everywheremeager} we prove Theorem \ref{t:reverse}, as well as study some problems concerning the possible generalizations of Corollary \ref{c-sAD} on the second level of the projective hierarchy. In Section \ref{stradrs} we prove Theorem \ref{t:martin}.

In Section \ref{mixrealssec} we study the notions of mixing and injective mixing reals. In this section we will heavily use standard facts about forcing notions, for the details see \cite{BaJu}.

Finally, in Section \ref{relques}, we list some open questions concerning our results.

\section{Descriptive set theory and ideals}\label{prelim}

As usual, $\Ubf{\Sigma}^0_\al, \Ubf{\Pi}^0_\al$ will stand for the $\alpha$th level of the Borel hierarchy while we denote by $\Ubf{\Sigma}^1_n, \Ubf{\Pi}^1_n$ the levels of the projective hierarchy. If $r$ is a real, the appropriate relativised versions are denoted by $\Sigma^0_\al(r), \Pi^0_\al(r)$, etc. For the ambiguous classes we write $\Ubf{\Delta}^{i}_\al$ and $\Delta^{i}_\al(r)$.

Suppose that $\mc{I}$ is an ideal on the set $X$. As mentioned before, if $X$ is countable then we can talk about complexity of ideals: $\mc{I}$ is $F_\sigma$, $\Ubf{\Sigma}^0_\al$, $\Ubf{\Pi}^1_n$, etc if $\mc{I}\subseteq  \mc{P}(X)\simeq 2^X$ is an $F_\sigma$, $\Ubf{\Sigma}^0_\al$, $\Ubf{\Pi}^1_n$, etc set in the usual compact Polish topology on $2^X$. If we fix a bijection between $\omega$ and $X$ we can define the collection of $\Sigma^0_\al(r), \Pi^0_\al(r)$, etc subsets  of $2^X$ as well. If $X=\omega^n,\Delta=\{(n,m)\in\om^2:m\leq n\},[\om]^n,2^{<\om},\om^{<\om},\mbb{Q}(=\{$rational numbers$\})$ then the we will always assume that the bijection is the usual, recursive one.

For example, $\mathrm{Fin}=[\om]^{<\om}$ is an $F_\sigma$ ideal, $\mc{Z}=\{A\subseteq\om:|A\cap n|/n\to 0\}$ is $F_{\sigma\delta}$, and $\mathrm{Conv}=\{A\subseteq\mbb{Q}\cap [0,1]:A$ has only finitely many accumulation points$\}$ is $F_{\sigma\delta\sigma}$, etc (see more examples in Section \ref{projid}). Similarly, we can associate descriptive complexity to any $\mc{X}\subseteq\mc{P}(\om)$, and we can also talk about the Baire property and measurability of subsets of $\mc{P}(\om)$. Clearly, if $Y\in\mc{I}^+$ then $\mc{I}\upharpoonright Y$ belongs to the same Borel or projective class in $\mc{P}(Y)$ as $\mc{I}$ in $\mc{P}(\om)$ (simply because $\mc{I}\upharpoonright Y$ is a continuous preimage of $\mc{I}$).

\smallskip

For a family $\mc{H} \subset 2^X$ we will denote by $\mrm{id}(\mc{H})$ the ideal generated by the sets in $\mc{H}$. We say that an ideal $\mc{I}$ on a countably infinite set $X$ is
\begin{itemize}
\item {\em tall} if every infinite subset of $X$ contains an infinite element of $\mc{I}$;
\item a {\em P-ideal} if for every sequence $A_n\in\mc{I}$ ($n\in\om$), there is an $A\in\mc{I}$ such that $A_n\subseteq^* A$, that is, $|A_n\setminus A|<\om$ for every $n$.

\end{itemize}

We will need the following two fundamental results of descriptive set theory (see e.g. in \cite{Jech}):

\begin{thm} {\em (Shoenfield Absoluteness Theorem)}
If $V\subseteq W$ are transitive models, $\om_1^W\subseteq V$, and $r\in \om^\om\cap V$, then $\Sigma^1_2(r)$ formulas are absolute between $V$ and $W$.
\end{thm}

\begin{cor} If $X\subseteq\mc{P}(\om)$ is an analytic or coanalytic set in the parameter $r\in\om^\om$, then the statement ``$X$ is an ideal'' is absolute for transitive models $V\subseteq W$ with $\om_1^W\subseteq V$ and $r\in V$.
\end{cor}
\begin{proof}
Let $\varphi(x,r)$ be a $\Sigma^1_1(r)$ or $\Pi^1_1(r)$ definition of $X$ ($r\in\om^\om$). Then the statement ``$X$ is an ideal'' is the conjunction of the following formulas (i) $\forall$ $a\in\mathrm{Fin}$ $\varphi(a,r)$, (ii) $\forall$ $x,y$ $ (x\nsubseteq y$ or $\neg\varphi(y,r)$ or $\varphi(x,r))$, and (iii) $\forall$ $x,y$ $(\neg\varphi(x,r)$ or $\neg\varphi(y,r)$ or $\varphi(x\cup y,r))$.
In particular, ``$X$ is an ideal'' is $\Pi^1_2(r)$ and hence we can apply the Shoenfield Absoluteness Theorem.
\end{proof}

\begin{thm} {\em (Mansfield-Solovay Theorem)}
If $A\nsubseteq L[r]$ is a $\Sigma^1_2(r)$ set, then $A$ contains a perfect subset.
\end{thm}
Other than these notions and results above, we will use descriptive set theoretic tools such as $\Gamma$-completeness, $\Gamma$-hardness, etc which can all be found in \cite{Kechris}.

Let $\mathrm{Tree}=\{T\subseteq \om^{<\om}:T$ is a tree$\}$ be the usual Polish space of all trees on $\om$ (a closed subset on $\mc{P}(\om^{<\om})$) and as usual, we denote by $[T]=\{x\in \om^\om:\forall$ $n$ $x\upharpoonright n\in T\}$ the {\em body} of $T$, i.e. the set of all branches of $T$.

\section{Examples of Borel and projective ideals}\label{projid}

There are many classical examples of Borel ideals. Here we present some of those that have easily understandable definitions, and the reader can see that these examples are motivated by a wide variety of backgrounds. For the important role of these ideals, especially in characterisation results, see \cite{Hrusak}.

\smallskip
\textbf{Some $F_\sigma$ ideals:}

\smallskip
{\em Summable ideals.} Let $h:\om\fv [0,\infty)$ be a function such that $\sum_{n\in\om}
 h(n)=\infty$. The {\em summable ideal associated to $h$} is
\[ \ical_h=\bigg\{A\subseteq\om:\sum_{n\in A} h(n)<\infty\bigg\}.\]

It is easy to see that a summable ideal $\ical_h$ is tall iff
$\lim_{n\fv\infty}h(n)=0$, and that summable ideals are $F_\sigma$ P-ideals.
The {\em classical summable ideal} is $\mc{I}_{1/n}=\mc{I}_h$ where $h(n)=1/(n+1)$, or $h(0)=1$ and $h(n)=1/n$ if $n>0$.
We know that there are tall $F_\sigma$ P-ideals which are not summable ideals: Farah's example (see \cite[Example 1.11.1]{Farah}) is the following ideal:
\[ \mc{I}_F=\bigg\{A\subseteq\om:\sum_{n<\om}\frac{\min\big\{n,|A\cap [2^n,2^{n+1})|\big\}}{n^2}<\infty\bigg\}.\]

The {\em eventually different ideals.} \[\mc{ED}=\Big\{A\subseteq\om\times\om:\limsup_{n\fv\infty}|(A)_n|<\infty\Big\}\]
where $(A)_n=\{k\in\om:(n,k)\in A\}$, and $\mc{ED}_\mathrm{fin}=\mc{ED}\upharpoonright\Delta$ where $\Delta=\{(n,m)\in\om\times\om:m\le n\}$. $\mc{ED}$ and $\mc{ED}_\mathrm{fin}$ are not P-ideals.

\smallskip
The {\em van der Waerden ideal:}
\[ \mc{W}=\big\{A\subseteq\om:A\;\text{does not contain arbitrary long arithmetic progressions}\big\}.\]
Van der Waerden's well-known theorem says that $\mc{W}$ is a proper ideal. $\mc{W}$ is not a P-ideal. For some set-theoretic results about this ideal see e.g. \cite{jana1} and \cite{jana2}.

\smallskip
The {\em random graph ideal:}
\[ \mathrm{Ran}=\mathrm{id}\big(\big\{\text{homogeneous subsets of the random graph}\big\}\big)\]
where the {\em random graph} $(\om,E)$, $E\subseteq [\om]^2$ is up to isomorphism uniquely determined by the following property: If $A,B\in[\om]^{<\om}$ are nonempty and disjoint, then there is an $n\in\om\setminus(A\cup B)$ such that $\{\{n,a\}:a\in A\}\subseteq E$ and $\{\{n,b\}:b\in B\}\cap E=\0$. A set $H\subseteq\om$ is ($E$-)homogeneous iff $[H]^2\subseteq E$ or $[H]^2\cap E=\0$. $\mathrm{Ran}$ is not a P-ideal.

\smallskip
The {\em ideal of graphs with finite chromatic number}:
\[\mc{G}_\mathrm{fc}=\big\{E\subseteq [\om]^2:\chi(\om,E)<\om\big\}.\]
It is not a P-ideal.

\smallskip
{\em Solecki's ideal:} Let $\mathrm{CO}(2^\om)$ be the family of clopen subsets of $2^\om$ (it is easy to see that $|\mathrm{CO}(2^\om)|=\om$), and let $\Omega=\{A\in\mathrm{CO}(2^\om):\lam(A)=1/2\}$ where $\lam$ is the usual product measure on $2^\om$. The ideal $\mc{S}$ on $\Omega$ is generated by $\{I_x:x\in 2^\om\}$ where $I_x=\{A\in\Omega:x\in A\}$.
$\mc{S}$ is not a P-ideal.

\smallskip
\textbf{Some $F_{\sigma\delta}$ ideals:}

\smallskip
{\em Density ideals.}
Let $(P_n)_{n\in\omega}$ be a sequence of pairwise disjoint finite subsets of $\om$ and let $\vec{\mu}=(\mu_n)_{n\in\om}$ be a
sequences of measures, $\mu_n$ is concentrated on $P_n$ such that $\limsup_{n\fv\infty}\mu_n(\om)>0$. The {\em density ideal generated by $\vec{\mu}$}
is
\[ \mc{Z}_{\vec{\mu}}=\Big\{A\subseteq\om:\lim_{n\fv\infty}\mu_n(A)=0\Big\}.\]
A density ideal $\mc{Z}_{\vec{\mu}}$ is tall iff $\max\{\mu_n(\{i\}):i\in P_n\}\xrightarrow{n\fv\infty}0$, and density ideals are $F_{\sigma\delta}$ P-ideals.
The {\em density zero ideal} $\mc{Z}=\big\{A\subseteq\om:\lim_{n\fv\infty}|A\cap n|/n=0\big\}$ is a tall density ideal because let $P_n=[2^n,2^{n+1})$ and $\mu_n(A)=|A\cap P_n|/2^n$. It is easy to see that $\mc{I}_{1/n}\subsetneq\mc{Z}$, and Szemer\'edi's famous theorem implies that $\mc{W}\subseteq\mc{Z}$ (see \cite{szemeredi}). The stronger statement $\mc{W}\subseteq\mc{I}_{1/n}$ is a still open Erd\H{o}s prize problem.

\smallskip
The {\em ideal of nowhere dense subsets of the rationals:} \[ \mathrm{Nwd}=\big\{A\subseteq\mbb{Q}:\mathrm{int}(\overline{A})=\0\big\}\]
where $\mathrm{int}(\cdot)$ stands for the interior operation on subsets of the reals, and $\overline{A}$ is the closure of $A$ in $\mbb{R}$. $\mathrm{Nwd}$ is not a P-ideal.

\smallskip
The {\em trace ideal of the null ideal}: Let $\mc{N}$ be the $\sigma$-ideal of subsets of $2^\om$ with measure zero (with respect to the usual product measure). The {\em $G_\delta$-closure} of a set $A\subseteq 2^{<\om}$ is $[A]_\delta=\big\{x\in 2^\om:\exists^\infty$ $n$ $x\upharpoonright n\in A\big\}$, a $G_\delta$ subset of $2^\om$. The trace of $\mc{N}$ is defined by
\[ \mathrm{tr}(\mc{N})=\big\{A\subseteq 2^{<\om}:[A]_\delta\in \mc{N}\big\}.\]
It is a tall $F_{\sigma\delta}$ P-ideal.

\smallskip
\textbf{Some tall $F_{\sigma\delta\sigma}$ (non P-)ideals:}

\smallskip
The ideal $\mathrm{Conv}$ is generated by those infinite subsets of $\mbb{Q}\cap [0,1]$ which are convergent in $[0,1]$, in other words
\[ \mathrm{Conv}=\big\{A\subseteq \mbb{Q}\cap [0,1]:|\text{accumulation points of $A$ (in}\;\mbb{R})|<\om\big\}.\]

The Fubini product of $\mathrm{Fin}$ by itself:
\[ \mathrm{Fin}\otimes\mathrm{Fin}=\big\{A\subseteq\om\times\om:\forall^\infty\;n\in\om\;|(A)_n|<\om\big\}.\]

\textbf{Some non-tall ideals:}

\smallskip
An important $F_\sigma$ ideal:
\[ \mathrm{Fin}\otimes\{\0\}=\big\{A\subseteq\om\times\om:\forall^\infty\;n\in\om\;(A)_n=\0\big\},\]
\indent and its $F_{\sigma\delta}$ brother (a density ideal):
\[ \{\0\}\otimes\mathrm{Fin}=\big\{A\subseteq\om\times\om:\forall\;n\in\om\;|(A)_n|<\om\big\}.\]

Applying the Baire Category Theorem, it is easy to see that there are no $G_\delta$ (i.e. $\Ubf{\Pi}^0_2$) ideals and we already presented many $F_\sigma$ (i.e. $\Ubf{\Sigma}^0_2$) ideals. In general, we have Borel ideals at arbitrary high levels of the Borel hierarchy:

\begin{thm} {\em (see \cite{Cal85} and \cite{Cal88})} There are $\Ubf{\Sigma}^0_\al$- and $\Ubf{\Pi}^0_\al$-complete ideals for every $\al\geq 3$.
\end{thm}

About ideals on the ambiguous levels of the Borel hierarchy see \cite{Eng}.

\smallskip
We also present some (co)analytic examples.

\begin{thm} {\em (see \cite[page 321]{zaf})}
For every $x\in\om^\om$ let $I_x=\{s\in \om^{<\om}:x\upharpoonright |s|\nleq s\}$ where $\leq$ is the coordinatewise ordering on every $\om^n$. Then the ideal on $\om^{<\om}$ generated by $\{I_x:x\in\om^\om\}$ is $\Ubf{\Sigma}^1_1$-complete.
\end{thm}

\begin{thm}
\label{t:complete}
The ideal of graphs without infinite complete subgraphs,
\[ \mc{G}_\mathrm{c}=\big\{E\subseteq [\om]^2:\forall\;X\in [\om]^\om\;[X]^2\nsubseteq E\big\}\]
is a $\Ubf{\Pi}^1_1$-complete (in $\mc{P}([\om]^2)$), tall,  non P-ideal.
\end{thm}
\begin{proof}
Tallness is trivial. If for every $n\in\om$, we define $E_n=\{\{k,m\}:k\leq n,m\not=k\}\in\mc{G}_\mrm{c}$ and $E_n\subseteq^* E\subseteq [\om]^2$ for every $n$, then $E$ contains a complete subgraph (see also in \cite{Meza}), hence $\mc{G}_\mrm{c}$ is not a P-ideal.

Let $\mathrm{WF}=\{T\in\mathrm{Tree}:[T]=\0\}$ be the $\Ubf{\Pi}^1_1$-complete set of well-founded trees. Furthermore, let $\mathrm{Tree}'$ be the family of those trees $T$ such that (i) every $t\in T$ is strictly increasing and  (ii) if $\{t\in T:n\in\mathrm{ran}(t)\}\ne\0$ then it has a $\subseteq$-minimal element ($n\in\om$). Then it is not hard to see that $\mathrm{Tree}'$ is also closed in $\mc{P}(\om^{<\om})$ hence Polish. Finally, let $\mathrm{WF}'=\{T\in\mathrm{Tree}':[T]=\0\}$, clearly, it is also $\Pi^1_1$.

We will construct Wadge-reductions $\mathrm{WF}\leq_W\mathrm{WF}'\leq_W \mc{G}_\mathrm{c}$.

\smallskip
$\mathrm{WF}\leq_W\mathrm{WF}'$: Fix an order preserving isomorphism $j$ between $\om^{<\om}$ and a $T_0\in\mathrm{Tree}'$. More precisely, for a $t=(k_0,k_1,\dots,k_{m-1})\in\om^{<\om}$ let $j(t)=(p^{1}_{k_0},p^{1}_{k_0}p^{2}_{k_1},\dots,p^{1}_{k_0}p^{2}_{k_1}\dots p^{m}_{k_{m-1}})$ where $p_i$ denotes the $i$th prime number. Then $j$ is one-to-one, order preserving, and $T_0=j[\om^{<\om}]$ is a tree containing strictly increasing sequences. To show that $T_0$ satisfies (ii), assume that $n\in\ran(j(t))$ for some $n\in\om$ and $t\in\om^{<\om}$. Then, by the definition of $j$, $n=p^{1}_{k_0}p^{2}_{k_1}\dots p^{m}_{k_{m-1}}$ where $s=(k_0,k_1,\dots,k_{m-1})\leq t$, and if $n\in\ran(j(t'))$ for some $t'\in\om^{<\om}$ then $s\leq t'$, hence $j(s)$ is $\subseteq$-minimal in $\{h\in T_0:n\in\ran(h)\}$.

The map $\mathrm{Tree}\to\mathrm{Tree}'$, $T\mapsto j[T]$ is a continuous reduction of $\mrm{WF}$ to $\mrm{WF}'$. Continuity is trivial, and also that $[T]=\0$ iff $[j[T]]\ne\0$, in other words, $T\in\mathrm{WF}$ iff $j[T]\in\mathrm{WF}'$.

\smallskip
$\mathrm{WF}'\leq_W\mc{G}_\mathrm{c}$: For every $T\in\mathrm{Tree}'$ let $E_T=\bigcup\{[\mathrm{ran}(t)]^2:t\in T\}$. We show that the function $T\mapsto E_T$ is continuous. If $u,v\in \big[[\om]^2\big]^{<\om}$ are disjoint then it is easy to see that the preimage of the basic clopen set $[u,v]=\{E\subseteq [\om]^2:u\subseteq E,v\cap E=\0\}\subseteq\mc{P}([\om]^2)$ is
\[ \big\{T\in\mathrm{Tree}':\big(\forall\;\{x,y\}\in u\;\exists\;t\in T\;x,y\in \mathrm{ran}(t)\big)\;\text{and}\;\big(\forall\;t\in T\;v\cap [\mathrm{ran}(t)]^2=\0\big)\big\}.\]
Although, as the collection of the sets satisfying the second part of the condition is a countable intersection of clopen sets, this set seems to be closed (and it is enough to prove that $\mc{G}_\mrm{c}$ is $\Ubf{\Pi}^1_1$-complete), actually, it is open in $\mathrm{Tree}'$: Let $m=\max(\cup v)+1$. Then
the set $\{T\in\mathrm{Tree}':\forall$ $t\in T$ $v\cap [\mathrm{ran}(t)]^2=\0\}$ is the intersection of $\mathrm{Tree}'$ and the clopen set (in $\mc{P}(\om^{<\om})$)
\[ \big[\0,\big\{t\in m^{\leq m}:t\;\text{is strictly increasing and}\;v\cap [\mathrm{ran}(t)]^2\ne\0\big\}\big].\]
The function $T\mapsto E_T$ is a reduction of $\mathrm{WF}'$ to $\mc{G}_\mathrm{c}$: Clearly, if $T\in\mathrm{Tree}'$ and $x\in [T]$ then $X=\mathrm{ran}(x)\in [\om]^\om$ shows that $E_T\notin\mc{G}_\mathrm{c}$ (i.e. $[X]^2\subseteq E$). Conversely, if $[X]^2\subseteq E_T$ and  $X=\{k_0<k_1<\dots\}$, then for every $n$ there is a $t_n\in T$ such that $k_n,k_{n+1}\in\mathrm{ran}(t_n)$, we can assume that $t_n$ is minimal in $\{s\in T:k_{n+1}\in\mathrm{ran}(s)\}$. It yields that $t_0\subseteq t_1\subseteq t_2\subseteq\dots$ is an infinite chain in $T$.
\end{proof}

In the following example, we show that a seemingly ``very'' $\Pi^1_2$ definition can also give us a $\Ubf{\Pi}^1_1$-complete ideal.

\begin{thm}
The ideal
\[ \mc{I}_0=\big\{A\subseteq\om\times\om:\forall\;X,Y\in [\om]^\om\;\exists\;X'\in [X]^\om\;\exists\;Y'\in [Y]^\om\;A\cap (X'\times Y')=\0\big\}\]
is a $\Ubf{\Pi}^1_1$-complete (in $\mc{P}(\om\times\om)$), tall, non P-ideal.
\end{thm}
\begin{proof}
Tallness is trivial because injective partial functions from $\om$ to $\om$ belong to $\mc{I}_0$.  The failure of the P property is also easy: Consider the sets $n\times\om \in \mc{I_0}$. If for some $A$ we have $n\times\om \subseteq^* A$ for every $n$ then every vertical section of $A$ is co-finite, and such a set is clearly $\mc{I}_0$-positive.

First we show that this ideal is $\Ubf{\Pi}^1_1$, for which the next claim is clearly enough. For $X,Y\in[\om]^\om$ define $T^\uparrow(X,Y)=\{(n,k)\in X\times Y: n<k\}$ and $T^\downarrow(X,Y)=\{(n,k)\in X\times Y: n>k\}$.

\begin{claim}
$A\in\mc{I}_0$ iff for every infinite $X$ and $Y$ the set $A$ does not contain $T^\uparrow(X,Y)$ or $T^\downarrow(X,Y)$.
\end{claim}
\begin{proof}[Proof of the Claim.]
The ``only if'' part is trivial. Conversely, assume that $A\notin\mc{I}_0$, i.e. there exist $X,Y \in [\omega]^\omega$ such that $A \cap (X' \times Y') \ne \emptyset$ for every $X' \in [X]^\omega$ and $Y' \in [Y]^\omega$.  Fix increasing enumerations $X=\{x_0<x_1<x_2<\dots\}$ and $Y=\{y_0<y_1<y_2<\dots\}$. By shrinking the sets $X$ and $Y$, we can assume that $x_0<y_0<x_1<y_1< \dots$, in particular $X \cap Y=\emptyset$. Consider the following coloring $c: [\omega]^2\to 2\times 2$: for $m<n$ let $c(m,n)=(\chi_A(x_m,y_n),\chi_A(x_n,y_m))$ where $\chi_A(x,y)=1$ iff $(x,y)\in A$.

Applying Ramsey's theorem, there exists an infinite homogeneous subset $S \subseteq \omega$. Let $S=Z\cup W$ be a partition into infinite subsets such that the elements of $Z$ and $W$ follow alternatingly in $S$. Then the elements of the sets $X'=\{x_m:m \in Z\}$ and $Y'=\{y_n: n \in W\}$ follow alternatingly in $\omega$ as well.

$S$ cannot be homogeneous in color $(0,0)$, otherwise $A \cap (X' \times Y')=\emptyset$ would hold. Similarly, if $S$ is homogeneous in color $(1,1)$ then $X' \times Y' \subset A$ and we are done. Now suppose that $S$ is homogeneous in color $(1,0)$ (for $(0,1)$ the same argument works). If $x_m \in X',y_n \in Y'$ and $x_m<y_n$ then $m<n$ because $Z\cap W=\0$. Hence by the homogeneity of $S$ we can conclude $(x_m,y_n)\in A$, so $T^\uparrow(X',Y')\subseteq A$.
\end{proof}

Now we show that $\mathcal{I}_0$ is $\Ubf{\Pi}^1_1$-complete. We will use (see \cite[27.B]{Kechris}) that the set
\[S=\big\{C \in \mathcal{K}(2^\omega): \forall\;x\in C\;\forall^\infty\;n\in\om\;x(n)=0\big\}\]
is $\Ubf{\Pi}^1_1$-complete where $\mc{K}(2^\om)$ stands for the family of compact subsets of $2^\om$ equipped with the Hausdorff metric, i.e. with the Vietoris topology, we know that $\mc{K}(2^\om)$ is a compact Polish space.

To finish the proof, we will define a Borel map $\mc{K}(2^\om)\to \mc{P}(\om\times\om)$, $C \mapsto A_C$ such that  $C \in S$ iff $A_C \in \mathcal{I}_0$. Fix an enumeration $\{s_m:m \in \omega\}$ of $2^{<\omega}$, for every $s\in 2^{<\om}$ define $[s]=\{x\in 2^\om:s\subseteq x\}$ (a basic clopen subset of $2^\om$), and let \[ A_C=\big\{(m,n):|s_m|>n,\;s_m(n)=1,\;\text{and}\;[s_m] \cap C \ne \emptyset\big\}.\]

For $C \in S$ we show that $A_C \in \mc{I}_0$. Let $X,Y \in [\omega]^\omega$ be arbitrary. If the set $\{m \in X:[s_m] \cap C =  \emptyset\}$ is infinite then we are done, since
\[A_C \cap \big(\big\{m \in X:[s_m] \cap C =  \emptyset\big\} \times Y\big)=\emptyset.\]
Otherwise, using the compactness of $C$ we can choose an $\{m_0<m_1<\dots\}=X' \in [X]^\omega$ and a convergent sequence $(x_{i})_{i\in\om}$ such that $x_i \in [s_{m_i}] \cap C$ for every $i$. If $x_i \to x $ then $x\in C\in S$ so $x(n)=0$ for every $n\geq n_0$ for some $n_0$. If $n \in Y\setminus n_0$ then for every large enough $i$ we have $n<|s_{m_i}|$ and $s_{m_i}(n)=x(n)=0$, hence the section $\{m:(m,n) \in (A_C \cap (X'\times Y))\}$ is finite. On the other hand, for a fixed $m$ if $|s_m|\leq n$ then $(m,n)\notin A_C$, therefore the section $\{n:(m,n) \in (A_C \cap (X'\times Y))\}$ is also finite. By an easy induction, one can define an $X''\in [X']^\om$ and a $Y''\in [Y]^\om$ such that $A_C\cap (X''\times Y'')=\0$.

Now we show that if $C  \not \in S$ then $A_C \not \in \mathcal{I}_0$. Let $x \in C$ be so that $Y=\{n:x(n)=1\}$ is infinite and let $X=\{m:x \in [s_m]\}$. Now clearly, if $(m,n) \in X \times Y$ then $(m,n) \in A_C$ if and only if $n<|s_m|$. In particular, for every $n \in Y$ the set $\{m \in X: (m,n) \not \in A_C\}$ is finite, and it clearly implies that the rectangle $X\times Y$ witnesses that $A_C\notin\mc{I}_0$.
\end{proof}

\begin{rem}
One can give an alternate proof of Theorem \ref{t:complete} constructing a Borel reduction of the set $C$ to $\mc{G}_\mathrm{c}$.
\end{rem}

\begin{thm}
There exist $\Ubf{\Sigma}^1_n$ and $\Ubf{\Pi}^1_n$-complete tall ideals for every $n \geq 1$.
\end{thm}
\begin{proof}
First we will construct $\Ubf{\Sigma}^1_n$-complete ideals. Let $\mc{J}$ be a tall Borel ideal, $\mc{A}$ be a perfect $\mc{J}$-AD family, and let $\mc{A}_n$ be a $\Ubf{\Sigma}^1_n$-complete subset of the Polish space $\mc{A}$. Define $\mc{I}_n=\mrm{id}(\mc{J}\cup\mc{A}_n)$, i.e. $\mc{I}_n$ is the ideal generated by $\mc{J}\cup\mc{A}_n$. Then $\mc{I}_n$ is a tall proper (because $\mc{A}_n$ is infinite) ideal. $\mc{I}_n$ is $\Ubf{\Sigma}^1_n$ because
\[ \mc{I}_n=\big\{X\subseteq\om:\exists\;k\in\om\;\exists\;(A_i)_{i<k}\in\mc{A}_n^k\;\;X\setminus \big(A_0\cup A_1\cup\dots\cup A_{k-1}\big)\in\mc{J}\big\}\]
In order to see that $\mc{I}_n$ is $\Ubf{\Sigma}^1_n$-complete, we know that if $B$ is a $\Ubf{\Sigma}^1_n$ set in a Polish space $\mc{X}$, then it can be reduced to $\mc{A}_n$ with a continuous map $f:\mc{X}\to \mc{A}\subseteq\mc{P}(\om)$, furthermore applying the trivial observation that $\mc{A}_n=\mc{I}_n \cap \mc{A}$, we obtain that this map is in fact  a reduction of $B$ to $\mc{I}_n$ as well.

Now we proceed with $\Ubf{\Pi}^1_n$ ideals. Again, there exists a $\Ubf{\Pi}^1_n$-complete set $\mc{B}_n \subseteq \mc{A}$. The previous argument gives that the ideal $\mc{I}'_n=\mrm{id}(\mc{J}\cup\mc{B}_n)$ is $\Ubf{\Pi}^1_n$-hard, so it is enough to prove that $\mc{I}'_n$ is $\Ubf{\Pi}^1_n$. In order to see this just notice that since $\mc{A}$ is an $\mc{J}$-AD-family, if $\mc{I}_0=\mrm{id}(\mc{J}\cup\mc{A})$ then we have
\[ X \in \mc{I}_0\setminus \mc{I}'_n\;\;\;\text{iff}\;\;\; X \in \mc{I}_0\;\text{and}\; \exists\;A\in \mc{A}\setminus \mc{B}_n\;A\cap X\in\mc{J}^+.\]
This implies, as $\mc{I}_0$ is clearly $\Ubf{\Sigma}^1_1$, that $\mc{I}_0 \setminus \mc{I}'_n$ is a $\Ubf{\Sigma}^1_n$ set, and hence $\mc{I}'_n$ is $\Ubf{\Pi}^1_n$ (here we used that $\mc{I}'_n \subseteq \mc{I}_0$).
\end{proof}
The idea of the above proof can be used to construct $\Ubf{\Sigma}^0_\alpha$-complete ideals for $\alpha \geq 3$ as well.

\section{Proof of Theorem \ref{main}}\label{mainsec}

\begin{proof}
Applying Corollary \ref{c-sAD}, we can fix perfect $\mc{I}$-AD families $\mc{A}_X$ on every $X\in\mc{I}^+$. The statement ``$\mc{A}_X$ is an $\mc{I}$-AD family'' is (at most) $\Ubf{\Pi}^1_2$ hence absolute because if $\mc{A}_X=[T]$ is coded by the perfect tree $T\in\mathrm{Tree}_2=\{T\subseteq 2^{<\om}:T$ is a tree$\}$ then ``$\mc{A}_X$ is an $\mc{I}$-AD family''$\equiv$
\[\forall\;x,y\in [T]\;\big(x\in\mc{I}^+\;\text{and}\;(x=y\;\text{or}\; x\cap y\in\mc{I})\big)\] where of course we are working on $2^\om$ and $(x\cap y)(n)=x(n)\cdot y(n)$ for every $n$.

For every $X,Y\in \mc{I}^+$ let $B(X,Y)=\{A\in \mc{A}_X:A\cap Y\in\mc{I}^+\}$. Then it is a continuous preimage of $\mc{I}^+$ (under $\mc{A}_X\to \mc{P}(\om)$, $A\mapsto A\cap Y$), hence if $\mc{I}$ is analytic then $B(X,Y)$ is coanalytic, and similarly, if $\mc{I}$ is coanalytic then $B(X,Y)$ is analytic.

Let $\ka=|\mf{c}^V|^W$ and fix an enumeration $\{X_\al:\al<\ka\}$ of the set $\mc{I}^+\cap V$ in $W$. Working in $W$, we will construct the desired $\mc{I}$-AD refinement $\{A_\al:\al<\ka\}$, $A_\al\subseteq X_\al$ by recursion on $\ka$. During this process, we will also define a sequence $(B_\al)_{\al<\ka}$ in $\mc{I}^+$.

Assume that $\{A_\xi:\xi<\al\}$ and $(B_\xi)_{\xi<\al}$ are done. Let $\ga_\al$ be minimal such that $B(X_{\ga_\al},X_\al)$ contains a perfect set. This property, namely, that an analytic or coanalytic set $H\subseteq\mc{P}(\om)$ contains a perfect set, is absolute because if it is analytic then ``$H$ contains a perfect subset'' iff ``$H$ is uncountable'' is of the form ``$\forall$ $f\in\mc{P}(\om)^\om$ $\exists$ $x$ $(x\in H$ and $x\notin\mathrm{ran}(f))$'' hence it is $\Ubf{\Pi}^1_2$; and if $H$ is coanalytic then ``$H$ contains a perfect set'' is of the form ``$\exists$ $T\in\mathrm{Tree}_2$ $(T$ is perfect and $\forall$ $x\in [T]$ $x\in H)$'' hence it is $\Ubf{\Sigma}^1_2$. In particular, $\ga_\al\leq\al$.
We also know that if $C$ is a perfect set coded in $V$, then in $W$ it contains $\ka$ many new elements: We know it holds for $2^\om$ e.g. because of the group structure on it, and we can compute new elements of $C$ along a homeomorphism between $C$ and $2^\om$ fixed in $V$.
Let
\[ B_\al\in B(X_{\ga_\al},X_\al)\setminus \big(V\cup \{B_\xi:\xi<\al\}\big)\;\;\text{be arbitrary,}\] and finally, let $A_\al=X_\al\cap B_\al\in\mc{I}^+$. We claim that $\{A_\al:\al<\ka\}$ is an $\mc{I}$-AD family (it is clearly a refinement of $\mc{I}^+\cap V$). Let $\al,\be<\ka$, $\al\ne\be$.

If $\ga_\al=\ga_\be=\ga$ then $B_\al,B_\be\in\mc{A}_{X_\ga}$ are distinct, and hence $A_\al\cap A_\be\subseteq B_\al\cap B_\be\in\mc{I}$ (actually, we can assume that it is finite).

If $\ga_\al<\ga_\be$, then because of the minimality of $\ga_\be$, we know that $B(X_{\ga_\al},X_\be)$ does not contain perfect subsets. It is enough to see that $B(X_{\ga_\al},X_\be)$ is the same set in $V$ and $W$, i.e. if $\psi(x,r)$ is a $\Sigma^1_1(r)$ or $\Pi^1_1(r)$ definition of this set then $\forall$ $x\in W$ $(\psi(x,r)\rightarrow x\in V$). Why? Because then $B_\al\notin B(X_{\ga_\al},X_\be)$ but $B_\al\in\mc{A}_{X_{\ga_\al}}$, hence it yields that $A_\al\cap A_\be\subseteq B_\al\cap X_\be\in\mc{I}$.

The set $K:=B(X_{\ga_\al},X_\be)$ is analytic or coanalytic and does not contain perfect subsets (neither in $V$ nor in $W$). Applying the Mansfield-Solovay theorem, we know that $K\subseteq L[r]$ ($r\in V$). We also know that $(L[r])^V\cap\mc{P}(\om)=(L[r])^W\cap \mc{P}(\om)$ holds because $\om_1^W\subseteq V$, hence $K^V=K^W$.
\end{proof}

\begin{rem}
It is natural to ask the following: Assume that $V\subseteq W$ are transitive models, $W$ contains new reals, and let $C$ be a perfect set coded in $W$. Does $C$ contain at least $|\mf{c}^V|^W$ many new elements in $W$? In other words: Does $|C^W\setminus V|^W\geq |\mf{c}^V|^W$ hold? Surprisingly, the answer is no! Moreover, it is possible that there is a perfect set of groundmodel reals in the extension, see \cite{VeWo}.
\end{rem}

\begin{rem}
What can we say about possible generalizations of Theorem \ref{main}, for example, can we weaken the condition on the complexity of the ideal? In general, this statement is false. Let $\varphi(x)$ be a $\Sigma^1_2$ definition of a $\Sigma^1_2$ (i.e. $\Delta^1_2$) prime P-ideal $\mc{I}$ in $L$. (How to construct such an ideal? Using a $\Delta^1_2$-good well-order $\leq$ on $\mc{P}(\om)$, by the most natural recursion, at every stage extending our family with a $\leq$-minimal element which can be added without generating $\mc{P}(\om)$ and also with a $\leq$-minimal pseudounion of the previous elements, avoiding universal quantification by applying goodness, we obtain such an ideal.) We cannot expect that $\varphi(x)$ defines an ideal in general but we can talk about the {\em generated} ideal: $x\in \mc{J}$ iff ``$\exists$ $y\in\mc{I}$ $x\subseteq y$'' which is $\Sigma^1_2$ too. If $r$ is a Sacks real over $L$, then $\mc{J}$ is still a prime P-ideal in $L[r]$ (see \cite[Lemma 7.3.48]{BaJu}) hence $\mc{J}^+\cap L$ does not have any $\mc{J}$-ADR's in $L[r]$.
\end{rem}

\section{On the existence of perfect $(\mc{I},\mrm{Fin})$-AD families}\label{everywheremeager}

First of all, we show that the reverse implication in the first part of Corollary \ref{c-sAD} does not hold.

\begin{exa}
The assumption that there is a perfect $(\mc{I},\mrm{Fin})$-AD family does not imply that $\mc{I}$ is meager: Fix a prime ideal $\mc{J}$ on $\om$. For every partition $P=(P_n)_{n\in\om}$ of $\om$ into finite sets, fix an $X_P\in [\om]^\om$ such that $A_P=\bigcup\{P_n:n\in X_P\}\in \mc{J}$ (notice that $\mc{J}$ cannot be meager); and let the ideal $\mc{I}$ on $2^{<\om}$ be generated by the sets of the form $A'_P=\bigcup\{2^k:k\in A_P\}$.

Clearly, the family $\{\{f\upharpoonright n:n\in\om\}:f\in 2^\om\}$ of branches of $2^{<\om}$ is a perfect AD family. We show that $\{f\upharpoonright n:n\in\om\}\in\mc{I}^+$. Notice that $\{\dom(s):s \in A'_P\}=A_P \in \mc{J}$ for every $P$. Thus, a set of the form $B_f=\{f\upharpoonright n:n\in\om\}$ cannot be an element of the ideal because $\{\dom(s):s\in B_f\}=\om$.

$\mc{I}$ is not meager: Assume the contrary, then by Theorem \ref{t-char} there exists a partition $Q=(Q_n)_{n\in\om}$ of $2^{<\om}$ into finite sets such that $\{n\in\om:Q_n\subseteq A\}$ is finite for every $A\in\mc{I}$. Then there is a partition $P=(P_n)_{n\in\om}$ of $\om$ into finite sets such that for every $n$ there is an $m$ with $Q_m\subseteq \bigcup\{2^k:k\in P_n\}$. We know that $A'_P\in\mc{I}$, a contradiction because $A'_P$ contains infinitely many $Q_m$'s.
\end{exa}

What can we say if there are perfect $(\mc{I},\mrm{Fin})$-AD families on every $X\in\mc{I}^+$? In this case we have only consistent counterexamples.

\begin{thm}
Assume that $\mf{b}=\mf{c}$. Then there is a non-meager ideal $\mc{I}$ on $\om$ such that there are perfect $(\mc{I},\mrm{Fin})$-AD families on every $X\in\mc{I}^+$.
\end{thm}
\begin{proof}
Let $[\om]^\om=\{X_\al:\al<\mf{c}\}$ and $\{$partitions of $\om$ into finite sets$\}=\{P_\al=(P^\al_n)_{n\in\om}:\al<\mf{c}\}$ be enumerations. We will construct the desired ideal $\mc{I}$ as an increasing union $\bigcup\{\mc{I}_\al:\al<\mf{c}\}$ of ideals by recursion on $\al<\mf{c}$. At the $\al$th stage we will make sure that
\begin{itemize}
\item[(i)] $\mc{I}_\al$ is generated by $|\al|$ many elements;
\item[(ii)] $P_\al$ cannot witness that $\mc{I}_\al$ is meager;
\item[(iii)] either $X_\al$ belongs to $\mc{I}_\al$ or there is a perfect $(\mc{I}_\al,\mrm{Fin})$-AD family on $X_\al$;
\item[(iv)] we do not destroy the $(\mc{I}_\be,\mrm{Fin})$-AD families we may have constructed in previous stages.
\end{itemize}

Let $\mc{I}_0=\mrm{Fin}$ and fix a perfect AD family $\mc{A}_0$ on $X_0$. At stage $\al>0$ we already have the ideals $\mc{I}_\be$ for every $\be<\al$, let $\mc{I}_{<\al}=\bigcup\{\mc{I}_\be:\be<\al\}$. We also have perfect $(\mc{I}_{<\al},\mrm{Fin})$-AD families $\mc{A}_\be$ on $X_\be\in\mc{I}_{<\al}^+$ for certain $\be\in D_\al\subseteq\al$.

If we can add $X_\al$ to $\mc{I}_{<\al}$, that is, $\mc{A}_\be\cap\mrm{id}(\mc{I}_{<\al}\cup\{X_\al\})=\0$ for every $\be\in D_\al$, then let $\mc{I}'_\al=\mrm{id}(\mc{I}_{<\al}\cup\{X_\al\})$ and $D'_\al=D_\al$.

Suppose that we cannot add $X_\al$ to $\mc{I}_{<\al}$, that is, $\mc{A}_\be\cap \mrm{id}(\mc{I}_{<\al}\cup\{X_\al\})\ne\0$ for some $\be\in D_\al$. Since
$\mc{I}_{<\al}$ is generated by $<\mf{b}=\mf{c}$ many sets, it is an everywhere meager ideal (see \cite{Solomon} or \cite[Thm. 9.10]{Blass}). We can apply Corollary \ref{c-sAD} to obtain a perfect $(\mc{I}_{<\al},\mrm{Fin})$-AD family $\mc{A}_\al$ on $X_\al$, let $\mc{I}'_\al=\mc{I}_{<\al}$, and let $D'_\al=D_\al\cup\{\al\}$.


Fix a partition $Q=(Q_n)_{n\in\om}$ of $\om$ into finite sets such that $\{n\in\om:Q_n\subseteq A\}$ is finite for every $A\in\mc{I}'_\al$ (we know that $\mc{I}'_\al$ is meager).

\begin{claim} There exist partitions $Q_{\be,B}=(Q^{\be,B}_n)_{n\in\om}$ for every $\be\in D'_\al$ and $B\in \mc{I}'_\al$ such that $A\cap Q^{\be,B}_n\setminus B\ne\0$ for every $\be\in D'_\al$, $A\in\mc{A}_\be$, $B\in\mc{I}'_\al$, and $n\in\om$.
\end{claim}
\begin{proof}[Proof of the Claim.] Let $\be\in D'_\al$ and $B\in\mc{I}'_\al$. We know that $\mc{A}_\be$ is compact as a subset of $\mc{P}(\om)$. Basic open sets in $\mc{P}(\om)$ are of the form $[s,t]=\{A\subseteq \om:s\cap A=\0$ and $t\subseteq A\}$ for disjoint, finite $s,t\subseteq\om$. Then $\mc{A}_\be\subseteq\bigcup\{[\0,\{n\}]:n\in\om\setminus B\}$ because $A\setminus B$ is infinite for every $A\in\mc{A}_\be$. Therefore $\mc{A}_\be\subseteq \bigcup\{[\0,\{n\}]:n\in N_0\setminus B\}$ for an $N_0\in\om$, in particular, $A\cap N_0\setminus B\ne\0$ for every $A\in\mc{A}_\be$. Let $Q^{\be,B}_0=[0,N_0)$. We can proceed by the same argument: $\mc{A}_\be\subseteq\bigcup\{[\0,\{n\}]:n\in [N_0,\om)\setminus B\}$ hence there is an $N_1>N_0$ such that $\mc{A}_\be\subseteq\bigcup\{[\0,\{n\}]:n\in [N_0,N_1)\setminus B\}$, in other words, $A\cap [N_0,N_1)\setminus B\ne\0$ for every $A\in\mc{A}_\be$. Let $Q^{\be,B}_1=[N_0,N_1)$ etc.
\end{proof}

Now we have the family $\mc{Q}=\{P_\al\}\cup\{Q\}\cup\{Q_{\be,B}:\be\in D'_\al,B \in \mc{C}_{\alpha}\}$ of partitions where $\mc{C}_\al\subseteq\mc{I}'_\al$ is a cofinal family, $|\mc{C}_\al|\leq\max\{|\al|,\om\}$. $|\mc{Q}|<\mf{c}=\mf{b}$ hence there is a partition $R=(R_m)_{m\in\om}$ which dominates all of these partitions, that is, $\forall$ $P=(P_n)_{n\in\om}\in\mc{Q}$ $\forall^\infty$ $m$ $\exists$ $n$ $P_n\subseteq R_m$ (see \cite[Thm. 2.10]{Blass}). Let $Y=\bigcup\{R_{2n}:n\in\om\}$ and $\mc{I}_{\al}=\mrm{id}(\mc{I}'_\al\cup\{Y\})$.

Then (i) is clearly satisfied, in order to see (ii) notice that by the fact that the partition $R_m$ was dominating and $P_\alpha \in \mc{Q}$, for almost every $m$ there exists an $n$ with $P^\alpha_n \subset R_{2m}$. Condition (iii) is also clear if $X_\alpha \in\mc{I}'_\al$.

If $X_\alpha \not \in \mc{I}'_\al$ then by definition $\alpha \in D'_\alpha$ so to see (iii) and (iv) we have to show that for every $\beta \in D'_\alpha$ the family $\mc{A}_\beta$ is not just an $(\mc{I}'_\al,\mrm{Fin})$-AD family, but also an $(\mc{I}_\al,\mrm{Fin})$-AD family. In other words, it is enough to check that for every $A \in \mc{A}_\beta$ and $B \in \mc{I}'_\al$ we have $A \setminus (B \cup Y) \ne\emptyset$. Fix such $A$ and $B$, we can assume that $B\in\mc{C}_\al$. Then for almost every $m$, there is an $n_m$ such that $Q^{\be,B}_{n_m}\subseteq R_{2m+1}$, and by the claim we know that $A\cap Q^{\be,B}_{n_m}\setminus B\ne\0$. Therefore, $A\setminus (B\cup Y)$ is infinite, hence $\mc{A}_\be\cap\mc{I}_\al=\0$ for every $\be\in D'_\al$.
\end{proof}

What can we say about ideals on the second level of the projective hierarchy, do there always exist perfect or at least uncountable $(\mc{I},\mrm{Fin})$-AD families? If all $\Ubf{\Sigma}^1_2$ and $\Ubf{\Pi}^1_2$ sets have the Baire property, then of course, yes because then $\Ubf{\Sigma}^1_2$ and $\Ubf{\Pi}^1_2$ ideals are meager and we can apply  Corollary \ref{c-sAD}. On the other hand, if $\mc{I}$ is a $\Sigma^1_2$ (i.e. $\Delta^1_2$) prime ideal (e.g. in $L$) then every $\mc{I}$-AD family is a singleton.

Similarly, we can construct a $\Sigma^1_2$-ideal $\mc{J}$ in $L$ such that there are infinite $\mc{J}$-AD families but all of them are countable: Copy the above ideal $\mc{I}$ to the elements of a partition $\{P_n:n\in\om\}\subseteq [\om]^\om$ of $\om$, and let $\mc{J}$ be the generated ideal.

This last example is very artificial in the sense that, this ideal is constructed from maximal ideals in a very ``obvious'' way, many of its restrictions are prime ideals. However, we can construct even more peculiar ideals:
\begin{prop} \label{strangeideal}
Suppose that there exists a $\Ubf{\Delta}^1_n$ prime ideal on $\omega$ for some $n$. Then there exists a $\Ubf{\Delta}^1_n$ ideal $\mc{I}$ such that $\mc{I}$ is nowhere maximal but every $\mc{I}$-AD family is countable. In particular, there exists such a $\Ubf{\Delta}^1_2$ ideal in $L$.
\end{prop}
\begin{proof}
Let $\mc{U}$ be an ultrafilter and define $\mu:\mc{P}(\om)\to [0,1]$ as $\mu(A)=\lim_\mc{U}\frac{|A\cap n|}{n}$ where $\lim_\mc{U}$ stands for the {\em $\mc{U}$-limit} operation on sequences in topological spaces, that is, $\lim_\mc{U}(a_n)=a$ iff $\{n\in\om:a_n\in V\}\in\mc{U}$ for every neighbourhood $V$ of $a$. It is easy to see that if $\overline{\{a_n:n\in\om\}}$ is compact, then $\lim_\mc{U}(a_n)_{n\in\om}$ exists, in particular, $\mu$ is defined on every $A\in\mc{P}(\om)$. It is also straightforward to show that $\mu$ is a finitely additive non-atomic probability measure on $\mc{P}(\om)$, that is, $\mu(\0)=0$, $\mu(A\cup B)=\mu(A)+\mu(B)$ if $A\cap B=\0$, $\mu(\om)=1$, and if $\mu(X)=\eps>0$ then for every $\delta\in (0,\eps)$ there is a $Y_\delta\subseteq X$ with $\mu(Y_\delta)=\delta$.

Let $\mc{I}=\{A\subseteq\om:\mu(A)=0\}$. Then $\mc{I}$ is an ideal. $\mc{I}$ is nowhere maximal because of $\mu$ is non-atomic (in particular, there are infinite $\mc{I}$-AD families). We show that every $\mc{I}$-AD family is countable. If there was an uncountable $\mc{I}$-AD  family $\mc{A}$, then $\mc{A}_n=\{A\in\mc{A}:\mu(A)>1/n\}$ would be uncountable for some $n\in\om$ and therefore among every $n$ many element of $\mc{A}_n$ there would be two with $\mc{I}$-positive intersection.

Notice that if $\mc{U}$ is $\Ubf{\Delta}^1_n$ ($n\geq 2$) then $\mc{I}$ is also $\Ubf{\Delta}^1_n$ because $A\in\mc{I}$ iff $\forall$ $k\in\om$ $\{n\in\om: |A\cap n|/n<2^{-k}\}\in\mc{U}$, and the function $A\mapsto \{n\in\om: |A\cap n|/n<2^{-k}\}$ is continuous (for every $k$).
\end{proof}

\section{On $(\mc{I},\mrm{Fin})$-ADR's}\label{stradrs}

In this section, we study Question \ref{str}.

\begin{thm}
Assume $\mrm{MA}_\ka$ and let $\mc{I}$ be an everywhere meager ideal, then every $\mc{H}\in [\mc{I}^+]^{\leq\ka}$ has an $(\mc{I},\mrm{Fin})$-ADR.
\end{thm}
\begin{proof}
Let $\mc{H}=\{H_\al:\al<\ka\}$ be an enumeration. Define $p\in\PP=\PP(\mc{H})$ iff $p$ is a function, $\dom(p)\in [\ka]^{<\om}$, and $p(\al)\in [H_\al]^{<\om}$ for every $\al\in\dom(p)$; $p\leq q$ iff $\dom(p)\supseteq \dom(q)$, $\forall$ $\al\in\dom(q)$ $p(\al)\supseteq q(\al)$, and $\forall$ $\{\al,\be\}\in[\dom(q)]^2$ $p(\al)\cap p(\be)=q(\al)\cap q(\be)$.

\smallskip
Then $\PP$ is a poset. First of all, we show that $\PP$ has the ccc. Let $\{p_\xi:\xi<\om_1\}\subseteq\PP$. Then $\{\dom(p_\xi):\xi<\om_1\}\subseteq [\ka]^{<\om}$. We can assume that this family forms a $\Delta$-system, $\dom(p_\xi)=D_\xi\cup R$. There are at most $\om$ many functions $R\to \mrm{Fin}$, hence we can also assume that there is a $q\in\PP$ such that $p_\xi\upharpoonright R= q$ for every $\xi<\om_1$. Clearly, $p_\xi\cup p_\zeta\in\PP$ and $p_\xi\cup p_\zeta\leq p_\xi$ for every $\xi,\zeta<\om_1$.

\smallskip
It is easy to see that for every $\al<\ka$ the set $D_\al=\{p\in\PP:\al\in\dom(p)\}$ is dense in $\PP$. If $G$ is a $\{D_\al:\al<\ka\}$-generic filter, then let $F_G:\ka\to\mc{P}(\om)$, $F_G(\al)=\bigcup\{p(\al):p\in G\}$. Clearly, $F_G(\al)\subseteq H_\al$ for every $\al$.

\smallskip
We show that $F_G(\al)\cap F_G(\be)$ is finite for every distinct $\al,\be<\ka$. Let $p\in D_\al\cap G$, $q\in D_\be\cap G$, and $r\in G$ be a common lower bound of them. It is easy to see that $F_G(\al)\cap F_G(\be)=r(\al)\cap r(\be)$.

\smallskip
If somehow we can make sure that $F_G(\al)\in\mc{I}^+$, then we are done because $\{F_G(\al):\al<\ka\}$ will be an $(\mc{I},\mrm{Fin})$-ADR of $\mc{H}$. We show that if $G$ is $(V,\PP)$-generic then $F_G(\al)$ is a Cohen-real in $\mc{P}(H_\al)$ over $V$. It is enough because then $F_G(\al)\notin \mc{I}\upharpoonright H_\al$ (we know that $\mc{I}\upharpoonright H_\al$ is meager) and to show that $V[F_G(\al)]\models F_G(\al)\notin \mc{I}\upharpoonright H_\al$, it is enough to use countable many dense sets. Why? For every $\al$ we can fix a countable family $\mc{C}_\al=\{C^\al_n:n\in\om\}$ of closed nowhere dense subsets of $\mc{P}(H_\al)$ which covers $\mc{I}\upharpoonright H_\al$, and hence have countable many dense subsets of the Cohen forcing such that if a filter is generic for this family then the generic real is not covered by any element of $\mc{C}_\al$. More precisely, we have to translate these dense subsets of the Cohen forcing to dense subsets in $\PP$, it can be done by applying the (inverse of the) projection $\PP\to\mbb{C}(H_\al)$ defined below.

Fix an $\al<\ka$, let $\mbb{C}(H_\al)=\{s:s$ is a finite partial function form $H_\al$ to $2\}$ where $s\leq t$ iff $s\supseteq t$ (then $\mbb{C}(H_\al)$ adds a Cohen subset of $H_\al$ over $V$), and define the map $e=e_\al:\PP\to\mbb{C}(H_\al)$ as follows:
\begin{itemize}
\item[(i)] $\dom(e(p))=\bigcup\{p(\be)\cap H_\al:\be\in \dom(p)\}$;
\item[(ii)] $e(p)(n)=1$ iff $n\in p(\al)$.
\end{itemize}

We show that $e$ is a {\em projection} (see e.g. \cite[page 335]{abraham}) , that is,
\begin{align*} (1)\;\, & e\;\text{is order-preserving, onto, and}\;e(\0)=\0;\\
(2)\;\, & \forall\;p\in\PP\;\forall\;s\in\mbb{C}(H_\al)\;\big(s\leq e(p)\rightarrow\exists\;p'\leq p\;e(p')=s\big).
\end{align*}
Clearly, $e(\0)=\0$. Assume that  $p \leq q$. Then clearly $\dom(e(p))\supseteq \dom(e(q))$. If $n \in \dom(e(q))$ and $n\in q(\al)\subseteq p(\al)$ then $e(q)(n)=e(p)(n)=1$; if  $n\in\dom(e(p))$ and $n\in q(\be)\setminus q(\al)$ for some $\be\ne\al$ then, as $p(\al)\cap p(\be)=q(\al)\cap q(\be)$, $n\in p(\be)\setminus p(\al)$ and hence $e(q)(n)=e(p)(n)=0$. This yields that  $e$ is indeed order preserving.

To show that $e$ is onto, we have to assume that $H_\al\subseteq\bigcup\{H_\be:\be\ne\al\}$ (and w.l.o.g. we can do so by extending $\mc{H}$ to be a cover of $\om$ and adding $\om$ as an element to $\mc{H}$). For an $s\in\mbb{C}(H_\al)$ define $p\in\PP$ as follows: Fix a finite $D\subseteq\ka$ containing $\al$ such that $\dom(s)\subseteq\bigcup\{H_\be:\be\in D\}$, let $\dom(p)=D$, and define $p(\al)=s^{-1}(1)$ and $p(\be)=\{n\in H_\be\cap H_\al:s(n)=0\}$. Then $e(p)=s$.

To show that $e$ satisfies (2), fix a $p\in\PP$, an $s\in \mbb{C}(H_\al)$, and assume that $s\leq e(p)$. Define $p'\in\PP$ as follows: For every $n\in J=(s\setminus e(p))^{-1}(0)$ pick a $\ga_n\in\ka\setminus\{\al\}$ such that $n\in H_{\ga_n}$. Let $\dom(p')=\dom(p)\cup \{\ga_n:n\in J\}$ and define $p'(\al)=p(\al)\cup s^{-1}(1)$, if $\be\in\dom(p')\setminus\{\al\}$ then $p'(\be)=p(\be)\cup\{n\in J:\be=\ga_n\}$. It is straightforward to see that $p'\in\PP$, $p'\leq p$, and $e(p')=s$.

\smallskip
We know that if $G$ is $(V,\PP)$-generic then $e[G]$ generates a $(V,\mbb{C})$-generic filter $G'$. Notice that the Cohen real defined from $G'$ is $F_G(\al)$, so we are done.
\end{proof}

Unfortunately, at this moment, we do not know whether we really needed Martin's Axiom in the previous theorem or it holds in $\mrm{ZFC}$. We show that if we attempt to construct a counterexample, that is, say a tall Borel ideal $\mc{I}$ and a family $\mc{H}\in [\mc{I}^+]^{<\mf{c}}$ without a $(\mc{I},\mrm{Fin})$-ADR, we have to be careful.
Let us define the following cardinal invariants of tall ideals on $\om$: The {\em star-additivity} of $\mc{I}$ is
\[ \mrm{add}^*(\mc{I})=\min\big\{|\mc{X}|:\mc{X}\subseteq\mc{I}\;\text{and}\;\nexists\;A\in \mc{I}\;\forall\;X\in\mc{X}\;X\subseteq^* A\big\},\]
the {\em Fodor number} of $\mc{I}$ is
\[ F(\mc{I})=\min\big\{|\mc{H}|:\mc{H}\subseteq\mc{I}^+\;\text{has no}\;\mc{I}\text{-ADR}\big\},\]
and the {\em star-Fodor number} of $\mc{I}$ is
\[ F^*(\mc{I})=\min\big\{|\mc{H}|:\mc{H}\subseteq\mc{I}^+\;\text{has no}\;(\mc{I},\mrm{Fin})\text{-ADR}\big\}.\]

Clearly, $\mc{I}$ is a P-ideal iff $\mrm{add}^*(\mc{I})>\om$.
Proposition \ref{refgen} says that $F(\mc{I})=\mf{c}$ whenever $\mc{I}$ is everywhere meager; and clearly, $F^*(\mc{I})\leq F(\mc{I})$.

\begin{fact}
If $\mrm{add}^*(\mc{I})< F(\mc{I})$ then $\mrm{add}^*(\mc{I})< F^*(\mc{I})$. If $\mrm{add}^*(\mc{I})=F(\mc{I})$ then $F(\mc{I})=F^*(\mc{I})$.
\end{fact}
\begin{proof}
Assume that $\mc{H}=\{H_\al:\al<\ka\}\subseteq\mc{I}^+$ where $\ka=\mrm{add}^*(\mc{I})<F(\mc{I})$. First fix an $\mc{I}$-ADR $\{A_\al:\al<\ka\}$ of $\mc{H}$ ($A_\al\subseteq H_\al$). Then for every $\al<\ka$ fix a $B_\al\in\mc{I}$ such that $A_\al\cap A_\be\subseteq^* B_\al$ for every $\be<\al$, and let $A'_\al=A_\al\setminus B_\al$. Then $\{A'_\al:\al<\ka\}$ is an $(\mc{I},\mrm{Fin})$-ADR of $\mc{H}$.  The second statement can be proved by the same argument.
\end{proof}

In particular, if $\mc{I}$ is an everywhere meager P-ideal and $F^*(\mc{I})<\mf{c}$, then $F^*(\mc{I})<F(\mc{I})$ hence  $\mrm{add}^*(\mc{I})<F(\mc{I})$ and so $\om_1\leq\mrm{add}^*(\mc{I})<F^*(\mc{I})<\mf{c}$, therefore  $\mf{c}\geq\om_3$.

\section{Mixing reals}\label{mixrealssec}

In this section, we study two closely related properties of forcing notions, one of which is slightly stronger then ``$[\om]^\om\cap V$ has an ADR in $V^\PP$''.

\begin{Def}
Let $\PP$ be a forcing notion. We say that an $f\in \om^\om\cap V^\PP$ is a {\em mixing real} over $V$ if  $|f[X]\cap Y|=\om$ for every $X,Y\in [\om]^\om\cap V$.
If $f$ is one-to-one, then we call it an {\em injective mixing real} or  {\em mixing injection}.
\end{Def}

Clearly, in the definition above, it is enough to require that $f[X]\cap Y\ne\0$ for every $X,Y\in [\om]^\om\cap V$.

\begin{prop}\label{mixchar}
Let $\PP$ be a forcing notion.
Then the following are equivalent:
\begin{itemize}
\item[(i)] There is a mixing real $f\in\om^\om\cap V^\PP$ over $V$.
\item[(ii)] There is an $f\in\om^\om\cap V^\PP$ such that $f[X]=\om$ for all $X\in [\om]^\om\cap V$.
\item[(iii)] There is a partition $(Y_n)_{n\in\om}$ of $\omega$ into infinite sets in $V^\PP$ such that $\forall$ $X\in [\om]^\om\cap V$ $\forall$ $n$ $|X\cap Y_n|=\om$.
\item[(iii)'] There is a partition $(Y_n)_{n\in\om}$ of $\omega$ into infinite sets in $V^\PP$ such that $\forall$ $X\in [\om]^\om\cap V$ $\forall$ $n$ $X\cap Y_n\ne\0$.
\end{itemize}
\end{prop}
\begin{proof}
(ii)$\rightarrow$(i) and (iii)$\leftrightarrow$(iii)' are trivial. (ii)$\leftrightarrow$(iii)' because let $Y_n=f^{-1}(n)$ (and vice versa). Finally, (i) implies (ii): Fix a partition $(C_n)_{n\in\om}$ of $\om$ into infinite sets in $V$ and let $g:\om\to\om$, $g\upharpoonright C_n\equiv n$. If $f$ is a mixing real over $V$, then $h=g\circ f$ has the required property.
\end{proof}

(iii) says that mixing reals can be seen as ``infinite splitting parititions''. Recall that a set $S\subseteq\om$ is a {\em splitting real} over $V$ if $|X\cap S|=|X\setminus S|=\om$ for every $X\in [\om]^\om\cap V$, in other words, $P=\{S,\om\setminus S\}$ is a partition of $\om$ such that $\forall$ $X\in [\om]^\om\cap V$ $\forall$ $Y\in P$ $|X\cap Y|=\om$.

Why is this property relevant to almost-disjoint refinements?
Fix an AD family $\mc{A}=\{A_\al:\al<\mf{c}\}$ in $V$, and let $\{X_\al:\al<\mf{c}\}$ be an enumeration of $[\om]^\om$ in $V$. If $f\in\om^\om\cap V^\PP$ is a mixing injection over $V$, then the family $\{f[A_\al]\cap X_\al:\al<\mf{c}\}\in V^\PP$ is an ADR of $[\om]^\om\cap V$.

\begin{prop}\label{i-iv} Let $\PP$ be a forcing notion.
\begin{itemize}
\item[(i)] If $\PP$ adds random  reals then it adds mixing reals.
\item[(ii)] If $\PP$ adds dominating reals, then it adds mixing reals.
\item[(iii)] If $\PP$ adds Cohen reals then it adds mixing injections.
\item[(iv)] If $\PP$ adds mixing injections then it adds unbounded reals.
\item[(v)] If $\PP$ has the Laver-property, then it does not add injective mixing reals.
\end{itemize}
\end{prop}
\begin{proof}
(i): Let $\lam$ be the usual probability measure on $\om^\om$, that is, $\lam$ is uniquely determined by the values $\lam([s])=2^{-s(0)-s(1)-\dots-s(n-1)-n}$ where $s:n\to\om$ and $[s]=\{f\in\om^\om:s\subseteq f\}$.  If $\mc{N}_\lam=\{A\subseteq\om^\om:\lam(A)=0\}$, then it is well-know that $\mrm{Borel}(\om^\om)/\mc{N}_\lam$ is forcing equivalent to the random forcing.
It is enough to see that the set $A_{X,Y}=\{f\in\om^\om:|f[X]\cap Y|<\omega\}$ is a null set in $\om^\om$ for every $X,Y\in [\om]^\om$: $A_{X,Y}=\bigcup_{n\in\om}\{f\in\om^\om:f[X]\cap Y\subseteq n\}$ and if $X=\{x_k:k\in\om\}$ and  $n\in\om$ then $\{f:f[X]\cap Y\subseteq n\}=\{f:\forall$ $k$ $f(x_k)\in n\cup (\om\setminus Y)\}$. Clearly, $\sum\{2^{-m-1}:m\in n\cup (\om\setminus Y)\}=\eps<1$ and hence $\lam(\{f:f[X]\cap Y\subseteq n\}) \leq \lim_{k\to\infty}\eps^k=0$.

\smallskip
(ii): Trivial modification of the proof of the fact (see e.g. \cite[Fact 20.1]{gentle}) that adding a dominating real implies adding a splitting real works here as well: Adding a dominating real is equivalant to adding a dominating partition $(P_n)_{n \in\om}$ of $\om$ into finite sets (see \cite[Thm. 2.10]{Blass}), that is, for every partition $(Q_m)_{m\in\om}\in V$ of $\om$ into finite sets, $\forall^\infty$ $n$ $\exists$ $m$ $Q_m\subseteq P_n$. Now any infinite partition of $\om$ containing of unions of infinitely many $P_n$'s satisfy (iii) from Proposition \ref{mixchar}.

\smallskip
(iii): We can talk about {\em injective Cohen-reals}. Simply consider the forcing notion $(\mathrm{Inj},\supseteq)$ where $\mathrm{Inj}=\{s\in\om^{<\om}:s$ is one-to-one$\}$, or the forcing notion $(\mathrm{Borel}(\mathrm{INJ})\setminus \mc{M}(\mathrm{INJ}),\subseteq)$ where $\mathrm{INJ}=\{f\in\om^\om:f$ is one-to-one$\}$ is a nowhere dense closed subset on $\om^\om$ and $\mc{M}(\mathrm{INJ})$ is the meager ideal on this Polish space.  It is not difficult to see that these forcing notions are forcing equivalent to the Cohen forcing (moreover, $\mrm{INJ}$ is homemomorphic to $\om^\om$).

If $c$ is an injective Cohen-real over $V$, then $c$ is mixing: For every $X,Y\in [\om]^\om$, the set  $A'_{X,Y}=A_{X,Y}\cap\mathrm{INJ}=\bigcup_{n\in\om}\big\{f\in\mathrm{INJ}:f[X]\cap Y\subseteq n\big\}$ is meager because $\{f\in\mathrm{INJ}:f[X]\cap Y\subseteq n\}$ is closed and nowhere dense in $\mathrm{INJ}$.

\smallskip
(iv): Let $f\in\mathrm{INJ}\cap V^\PP$ be a mixing injection and assume on the contrary that there is a strictly increasing $g\in \om^\om\cap V$ such that $f,f^{-1}<g$ (where of course $f^{-1}<g$ means that $f^{-1}(k)<g(k)$ for every $k\in\mathrm{ran}(f)$).

We define $X=\{x_k:k\in\om\},Y=\{y_k:k\in\om\}\in [\om]^\om$ in $V$ as follows: $x_0=0$, $y_0=g(0)$, $x_n=\max\{g(y_k):k<n\}$, and $y_n=g(x_n)$. Suppose that  $f(x_k)=y_l$ for some $k,l \in \om$. If $k \leq l$ then   \[f(x_k)<g(x_k) = g\big(\max_{m<k}g(y_m)\big) \leq g\big(\max_{m<l}g(y_m)\big)=x_l<g(x_l)=y_l,\] a contradiction. Now, if $k > l$ then \[x_k=f^{-1}(y_l)<g(y_l) \leq \max\{g(y_m):m<k\} =x_k\] which is again impossible. Thus, $f[X]\cap Y=\0$, so $f$ cannot be a mixing injection.

\smallskip
(v): Fix a sequence $(a_n)_{n\in\om}\in\om^\om\cap V$ satisfying $a_{n+1}-a_n>(n+2)2^{n+1}$ and $a_0>1$. Assume that $p\vd \dot{f}\in\mrm{INJ}$. Let $\dot{g}$ be a $\PP$-name for a function on $\om$ such that $p\vd \dot{g}(n)=\dot{f}\cap (a_n\times a_n)=\{(k,l)\in a_n\times a_n:f(k)=l\}$ for every $n$ (in particular, $p\vd$``$\dot{g}(n)$ is an injective partial function from $a_n$ to $a_n$''). Then $p\vd\dot{g}\in\prod_{n\in\om}\mc{P}(a_n\times a_n)$ hence, applying the Laver property of our forcing notion to the name $\dot{g}$ for a function from $\omega$ to $[\om\times\om]^{<\om}$, there is a $q\leq p$ and a ``slalom'' $S:\om\to \big[[\om\times\om]^{<\om}\big]^{<\om}$ in $V$ which catches $\dot{g}$, that is, $S(n)\subseteq \mathcal{P}(a_n\times a_n)$, $|S(n)|\leq 2^n$, and  $q\vd\dot{g}(n)\in S(n)$ for every $n$. Without loss of generality we can assume that all elements of $S(n)$ are injective partial functions $a_n\to a_n$.

Working in $V$, we will define the sets $X=\{x_n:n\in\om\},Y=\{y_n:n\in\om\}\in [\om]^\om$ by recursion on $n$ such that $q\vd \dot{f}[X]\cap Y=\0$.

Let $x_0\in a_0$ be arbitrary. We know that there is a $y_0\in a_0$ such that $(x_0,y_0)\notin \bigcup S(0)$ (a function cannot cover $\{(x_0,k):k<a_0\}$).

Assume that we already have $X_n=\{x_k:k\leq n\}$ and $Y_n=\{y_k:k\leq n\}$ such that $(X_n\times Y_n)\cap \bigcup_{k\leq n}\bigcup S(k)=\0$. There is an $x_{n+1}\in a_{n+1}\setminus a_n $ such that
\[ \big\{s(x_{n+1}):s\in S(n+1),x_{n+1}\in\dom(s)\big\}\cap Y_n=\0.\]
Why? If for every $m\in a_{n+1}\setminus a_n$ there is an $s_m\in S(n+1)$ such that $s_m(m)\in Y_n$ then there is a set $H\in [a_{n+1}\setminus a_n]^{n+2}$ such that $s_m=s$ does not depend on $m\in H$ (because $|a_{n+1}\setminus a_n|>(n+2)2^{n+1}$ and $|S(n+1)|\leq 2^{n+1}$). But it would mean that $H\subseteq \dom(s)$ and $|s[H]|\leq |Y_n|=n+1$ which is a contradiction because $s$ is injective.

We also want to fix a $y_{n+1}\in a_{n+1}\setminus a_n$ such that
$y_{n+1}\ne s(x_k)$ for any $k\leq {n+1}$, $s\in S(n+1)$ if $x_k\in\dom(s)$. The set of forbidden values is of size at most $2^{n+1}(n+2)$ hence there is such a $y_{n+1}$.
\end{proof}

In the diagram below, we summarize logical implications between classical properties of forcing notions and the ones we defined above. We will show that arrows without an $\ast$ above them are strict (i.e. not equivalences), and that there are no other implications between these properties. The arrow $\rDashto$ with question mark means that we do not know whether this implication holds (but the reverse implication is false). Of course, $\mbb{C}$ stands for the Cohen forcing, $\mbb{B}$ is the random forcing, and to keep the diagram small, we did not put ``$\PP$ adds \dots'' and ``$\PP$ has the \dots'' before the properties we deal with.

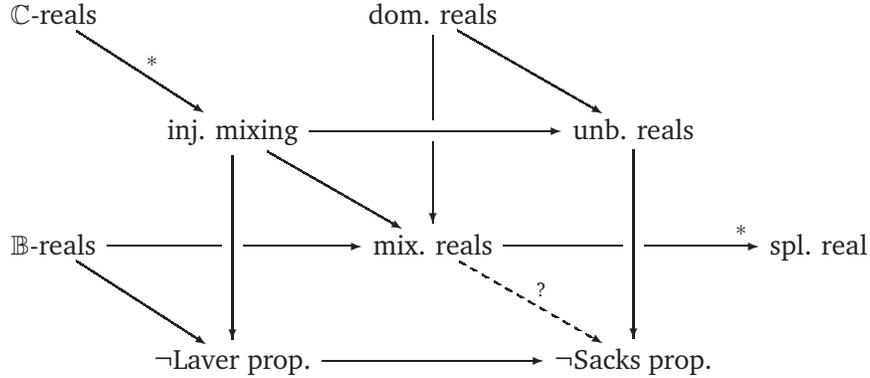
\begin{diagram}
\mbb{C}\text{-reals} & & & & \text{dom. reals} &&&&\\
 & \rdTo^{\ast} & && \vLine &\rdTo&&&\\
&& \text{inj. mixing} &&\HonV &\rTo& \text{unb. reals}&&\\
&& &\rdTo&\dTo&&&&\\
\mbb{B}\text{-reals}& \hLine & \VonH& \rTo & \text{mix. reals} & \hLine& \VonH & \rTo^{\ast} & \text{spl. real} \\
&\rdTo&\dTo&&& \rdDashto^{\text{?}} & \dTo&\\
&& \neg\text{Laver prop.} & & \rTo& & \neg\text{Sacks prop.} &\\
\end{diagram}

\medskip
The non-trivial non-implications in the diagram are the following:

\begin{itemize}
\item $\neg$Laver prop. $\nrightarrow$ splitting reals: The infinitely equal forcing $\mbb{EE}$ is $\om^\om$-bounding, preserves P-points (hence cannot add splitting reals), and $\vd_{\mbb{E}}$``$2^\om\cap V$ is a null set'' (see \cite[Lemma 7.4.13-15]{BaJu}). $\mbb{EE}$ cannot have the Laver property because otherwise it would have the Sacks property as well but then it could not force $2^\om\cap V$ to be of measure zero (it follows from e.g. \cite[Thm. 2.3.12]{BaJu}).
\item unbounded reals $\nrightarrow$ splitting reals: The Miller forcing (see \cite[7.3.E]{BaJu}).
\item spl. reals $\nrightarrow$ $\neg$Sacks prop.: The Silver forcing adds splitting reals (see \cite[Lemma 2.3]{gentle}) and it is straightforward to show that it satisfies the Sacks property.
\end{itemize}

We list the remaining questions in the next section.

\section{Related questions}\label{relques}

We already presented $\Ubf{\Sigma}^1_n$- and $\Ubf{\Pi}^1_n$-complete ideals but our construction was pretty artificial.

\begin{que}
Can we define ``natural'' $\Ubf{\Sigma}^1_n$- and $\Ubf{\Pi}^1_n$-complete ideals?
\end{que}

\begin{que}\label{strref}
Assume that $V,W$ and $\mc{I}$ are as in Theorem \ref{main}. Does there exist an $(\mc{I},\mrm{Fin})$-ADR of $\mc{I}^+\cap V$ in $W$? Or at least an $\mc{I}$-ADR $\{A_X:X\in \mc{I}^+\cap V\}\in W$ such that for every distinct $X,Y\in\mc{I}^+\cap V$ (using the notiations from the proof of Theorem \ref{main}) there is a $B_{X,Y}\in\mc{I}\cap V$ such that $A_X\cap A_Y\subseteq B_{X,Y}$?
\end{que}

\begin{que}
Does there exist a non-meager ideal $\mc{I}$ (in $\mrm{ZFC}$) such that there are perfect $(\mc{I},\mrm{Fin})$-AD families on every $X\in\mc{I}^+$?
\end{que}

In Example \ref{strangeideal}, assuming that there is a $\Ubf{\Delta}^1_2$ ultrafilter, we constructed a $\Ubf{\Delta}^1_2$ ideal $\mc{I}$ such that every $\mc{I}$-AD family is countable but $\mc{I}$ is nowhere maximal.

\begin{que}\label{ufbolsigma}
Is it consistent that there are no $\Ubf{\Delta}^1_2$ ultrafilters but there is a $\Ubf{\Sigma}^1_2$ ideal $\mc{I}$ such that every $\mc{I}$-AD family is countable but $\mc{I}$ is nowhere maximal?
\end{que}

A remark to Question \ref{ufbolsigma}: We know (see \cite[Thm. 9.3.9 (2)]{BaJu}) that if there are no dominating reals over $L[r]$ for any $r\in \om^\om$, then there is a $\Sigma^1_2$ unbounded hence non-meager filter. If every $\Ubf{\Delta}^1_2$ set is Lebesgue measurable or has the Baire property, then there are no $\Ubf{\Delta}^1_2$ ultrafilters. For instance,  these conditions above hold  in the Cohen and random models over $V=L$ (see \cite[Thm. 9.2.1]{BaJu}). In these models a non-meager $\Sigma^1_2$ ideal $\mc{I}$ must be nowhere maximal (otherwise a restriction of $\mc{I}$ would be a $\Delta^1_2$ prime ideal). It would be interesting to know the possible sizes of $\mc{I}$-AD families in these models.

\begin{que}
Is it consistent that for some (tall) Borel (P-)ideal $\mc{I}$ a family $\mc{H}\in [\mc{I}^+]^{<\mf{c}}$ does not have an $(\mc{I},\mrm{Fin})$-ADR (i.e. $F^*(\mc{I})<\mf{c}$)?
\end{que}

\begin{que}\label{quemix}
Does adding mixing injections imply adding Cohen reals?
\end{que}

\begin{que}
Does the Sacks property of a forcing notion imply that it does not add mixing reals?
\end{que}

Proposition \ref{mixchar} motivates the following notion: Let $n\geq 2$. We say that a forcing notion adds an {\em $n$-splitting partition}, if there is a partition $(Y_k)_{k<n}$ of $\om$ into infinite sets in $V^\PP$ such that $|X\cap Y_k|=\om$ for every $X\in [\om]^\om\cap V$ and $k<n$. In particular, adding $2$-splitting partitions is the same as adding splitting reals, and adding $\om$-splitting (infinite splitting) partitions is equivalent to adding mixing reals.

It is easy to see that if $\PP$ adds a splitting real then the $n$ stage iteration of $\PP$ adds a $2^n$-splitting partition. In fact, splitting reals and $n$-splitting partitions cannot be separated in terms of cardinal invariants. Let us denote $\mf{s}_n$ ($2\leq n<\om$) the least size of a family $\mc{S}_n$ of partitions of $\om$ into $n$ many infinite sets such that
\[\tag{$\ast$} \forall\;X\in [\om]^\om\;\exists\;P=(P_k)_{k<n}\in\mc{S}_n\;\forall\;k<n\;|X\cap P_k|=\om.\]
Of course, this definition makes sense for $n=\om$ as well but $\mf{s}_\om$ stands for an already defined and studied cardinal invariant. To avoid confusions, let us denote this cardinal by $\mf{s}_\mrm{mix}$.

Then $\mf{s}_n=\mf{s}=\mf{s}_2$ for every $2\leq n<\om$. For the non-trivial direction, assume that we have a family $\mc{S}$ of splitting partitions of size $\mf{s}$ and consider all possible ``$(n-1)$-long iterated nestings'' of these partitions. For example, if $n=3$ then to every pair $(P=(P_0,P_1),Q=(Q_0,Q_1))$ of partitions from $\mc{S}$ we associate a partition of $\om$ into three infinite sets as follows: Let $e_0:\om\to Q_0$ be the increasing bijection and take the partition $(e_0[P_0],e_0[P_1],Q_1)$. We obtain $\mf{s}^{n-1}=\mf{s}$ many partitions of $\om$ into $n$ many infinite sets, the family $\mc{S}_n$ of these partitions satisfies $(\ast)$, and hence $\mf{s}_n\leq\mf{s}$.

\begin{que}
Does adding $n$-splitting partitions ($2\leq n<\om$) imply adding $(n+1)$-splitting partitions?
\end{que}

\begin{que}
Is $\mf{s}_\mrm{mix}=\mf{s}$? Does adding splitting reals (or $n$-splitting partitions for every $n$) imply adding mixing reals? What can we say about the Silver forcing? (It is straightforward to see that it adds $n$-splitting partitions for every $n$.)
\end{que}

\end{document}